\documentclass[12pt,draftcls,onecolumn]{IEEEtran}
\pdfoutput=1
\usepackage{amssymb,mathrsfs,graphicx,subfigure, enumerate}
\usepackage{amsmath,amsfonts,amssymb,amscd,amsthm, bbm}
\usepackage{extpfeil}
\usepackage{graphicx,colortbl}
\usepackage{float}
\usepackage{epsfig}
\usepackage{caption}
\usepackage{color}
\usepackage{bm}
\graphicspath{{Figure/}}
\usepackage{todonotes}
\usepackage{verbatim}

\begin{document}

\title{Cucker-Smale model with a bonding force and a singular interaction kernel}
\author{Jeongho Kim and Jan Peszek}

\newtheorem{theorem}{Theorem}[section]
\newtheorem{lemma}{Lemma}[section]
\newtheorem{corollary}{Corollary}[section]
\newtheorem{proposition}{Proposition}[section]
\newtheorem{remark}{Remark}[section]
\newtheorem{definition}{Definition}[section]

\renewcommand{\theequation}{\thesection.\arabic{equation}}
\renewcommand{\thetheorem}{\thesection.\arabic{theorem}}
\renewcommand{\thelemma}{\thesection.\arabic{lemma}}
\newcommand{\bbr}{\mathbb R}
\newcommand{\bbz}{\mathbb Z}
\newcommand{\bbn}{\mathbb N}
\newcommand{\bbs}{\mathbb S}
\newcommand{\bbp}{\mathbb P}
\newcommand{\ddiv}{\textrm{div}}
\newcommand{\bn}{\bf n}
\newcommand{\rr}[1]{\rho_{{#1}}}
\newcommand{\thh}{\theta}
\def\charf {\mbox{{\text 1}\kern-.24em {\text l}}}
\renewcommand{\arraystretch}{1.5}


\maketitle
\footnote{The work of J. Peszek was supported by the Polish MNiSW grant Mobilno\' s\' c Plus no. 1617/MOB/V/2017/0 and partially supported by the Polish MNiSW grant Iuventus Plus no. 0888/IP3/2016/74 and the work of J. Kim was supported by the German Research Foundation (DFG) under the project number IRTG 2235.}
\footnote{J. Kim was with the Department of mathematical sciences, Seoul National University, 1 Gwanak-ro, Gwanak-gu, Seoul, 08826 Republic of Korea (e-mail: jhkim206@snu.ac.kr, phone: +82 2 880 1491)}
\footnote{J. Peszek is with the Center for Scientific Computation and Mathematical Modeling (CSCAMM),
University of Maryland, 4146 CSIC Building no. 406, 8169 Paint Branch Drive, College Park, MD 20742-3289, USA (e-mail: j.peszek@mimuw.edu.pl, phone: 301-405-1646, fax: 301-314-6674)
and also with Institute of Mathematics of the Polish Academy of Sciences, Warszawa, Poland }
\footnote{This work has been submitted to the IEEE for possible publication. Copyright may be transferred without notice, after which this version may no longer be accessible}

\date{\today}

\begin{abstract}
We prove the lack of asymptotic collisions between particles following the Cucker-Smale flocking model with a bonding force and its simplification. Moreover, we prove that in the case of the CSB model with a singular communication weight, finite-in-time collisions are impossible. Consequently, we establish existence of the global-in-time minimal distance between the particles. Furthermore, we show that asymptotic distribution of particles is confined within a ball of a given radius.
\end{abstract}

\section{Introduction}\label{sec:1}
\setcounter{equation}{0}
Emergence of pattern formation is an ubiquitous phenomenon observed in the collective behavior of ensembles of
self-propelled particles, e.g., flocking of birds or herding of sheep. In this paper, we use the jargon ``{\it flocking}" to describe such a collective dynamics. More precisely, ``{\it flocking}" represents a phenomenon in which self-propelled particles organize into an ordered motion using only limited environmental
information and simple rules \cite{T-T}. Recently, several mathematical models on the flocking phenomena have appeared in the literature \cite{C-S2, J-K, V-C-B-C-S}. They have been extensively studied owing to their potential engineering applications in sensor networks, the formation control of robots and unmanned aerial vehicles, etc. \cite{L-P-L-S, P-L-S-G-P, P-E-G}. 
Among such models, the Cucker-Smale (C-S) model describes the dynamics of particles with all-to-all interaction to align their velocities \cite{C-S2, H-Liu, H-T}. The research on the C-S model branches in various directions that are based on the applicational character of the model and thus are often qualitative in nature. Such directions include collision-avoidance \cite{A-C-H-L, C-C-M-P, C-D}, asymptotic behavior and pattern formation \cite{C-F-R-T, P-E-G,T-B}. Furthermore, the issue of collision-avoidance sparks the study on the C-S model with a singular kernel \cite{C-C-M-P, M-P,P1, P2}. On the other hand, the dynamics of the original C-S model exhibits only the property of velocity alignment and, in particular, there is no information about asymptotic pattern formation. Thus to enforce specific pattern formation, additional forces have been implemented \cite{ A-C-H-L,H-K-P}. For further information we refer to \cite{C-F-T-V, C-H-L}.

Particularly in \cite{H-K-P}, the authors adjust the C-S model by introducing a {\it bonding force}, controlling the distances between the particles to obtain dynamics in which particles do not collide or disperse asymptotically (and, in a sense, are bonded with each other). This C-S model with bonding force (CSB model) is governed by the following system
\begin{align}\label{B-2}
\left\{
\begin{array}{clc}
\displaystyle\frac{dx_i}{dt} &= v_i,\quad x_i,v_i\in\bbr^d,\quad i=1,2,\cdots,N,\quad t>0,\\
\displaystyle\frac{dv_i}{dt} &= \displaystyle\frac{K_1}{N}\displaystyle\sum_{j=1}^N\psi(|x_j-x_i|) (v_j-v_i)\\
&+\displaystyle\frac{\tilde{K}}{N}\displaystyle\sum_{j=1}^N \displaystyle\frac{1}{2r^2_{ij}}\left[(v_i-v_j) \cdot (x_i-x_j)\right](x_j-x_i)\\
&+\displaystyle\frac{K_2}{N} \displaystyle\sum_{j=1}^N\displaystyle\frac{r_{ij}-2R}{2r_{ij}}(x_j-x_i),\qquad r_{ij}:=|x_i-x_j|,
\end{array}
\right.
\end{align}
where $N$ is the number of particles, $(x_i,v_i)$ denotes position and velocity of $i$th $d$-dimensional particle, constants $K_1$, $K_2$ and $\tilde{K}$ control the intensity of the interaction and, finally, constant $R$ influences the asymptotic distance between the particles. The communication weight $\psi:\mathbb{R}_+\to \mathbb{R}_+$ is generally a non-increasing, smooth function. The CSB model differs from the original Cucker-Smale model by the addition of the latter two terms in $\eqref{B-2}_2$, which together compose the bonding force.

At the first glance the bonding force in \eqref{B-2} forces an asymptotic pattern in which distance between the particles converges to $2R$. However, when the number of the particles is much larger than the dimension of the space, this pattern formation is physically impossible. Instead, we observe in numerical simulations, that the particles move towards an ``energy minimizing" configuration. When the numerical simulation for CSB system \eqref{B-2} is implemented, we find that the particles converge to a pattern, which is characterized by a uniform spread of the particles in a ball of radius $2R$ (see Figure \ref{Fig1}). In particular one observes the lack of asymptotic collisions between the particles, even though, until now, asymptotic collision-avoidance was not proven mathematically.

\noindent \textbf{Main goal.} The main goal of this paper is to introduce a simplification of the CSB model that admits a global-in-time minimal distance between the particles. The simplified CSB system reads as follows:

\begin{align}\label{B-3}
\left\{
\begin{array}{clc}
\displaystyle\frac{dx_i}{dt} &= v_i,\quad x_i,v_i\in\bbr^d,\quad i=1,2,\cdots, N,\quad t>0,\\
\displaystyle\frac{dv_i}{dt} &= \displaystyle\frac{K_1}{N}\displaystyle\sum_{j=1}^N\psi(|x_j-x_i|) (v_j-v_i)\\
&+\displaystyle\frac{K_2}{N} \displaystyle\sum_{j=1}^N\displaystyle\frac{r_{ij}-2R}{2r_{ij}}(x_j-x_i).
\end{array}
\right.
\end{align}

The main motivation for our research comes from the applications in robotic multi-agent systems. The issue of collision avoidance is widely studied from the engineering point of view (e.g. \cite{A-A-E}). Existence of a minimal distance between the agents seems especially  important for a safe operation of unmanned aerial vehicles. On the other hand simplification of the system itself may reduce its computational complexity.

The main mathematical contribution of our work is the proof of asymptotic collision-avoidance and asymptotic bound on position of the particles for the CSB system (original \eqref{B-2} and simplified \eqref{B-3}). It bridges the gap between numerical simulations and theoretical knowledge for the CSB system. Furthermore we expand the ideas introduced in \cite{C-C-M-P}, proving finite-time collision avoidance, provided that the communication weight $\psi$ is sufficiently singular.

In our considerations, we assume that $\psi$ is a singular communication weight of the form
\begin{align}\label{sing}
\psi(s) = s^{-\alpha}, \quad \alpha\geq 1.
\end{align}
However, most of our results remain true also in case of regular $\psi$. Additionally we do not apply the specific form of $\psi$ given by \eqref{sing} and our argumentation can be easily generalized to any $\psi$ that is not integrable near zero. Finally, let us note that our methods strongly rely on finiteness of the number of the particles $N$ and do not provide too much information on the kinetic or hydrodynamic limit as $N\to\infty$.

The remainder of this paper is organized as follows. In Section \ref{sec:2}, we review some of the standard facts, most notably the total energy estimate. Section \ref{mr} contains precise statement of the main results and proofs. Finally, Section \ref{sec:5} presents several numerical experiments supporting our analysis.\\

\noindent\textbf{Notation.} For given families of vectors $\{x_i\}_{i=1}^N$ and $\{v_i\}_{i=1}^N$ in $\bbr^d$, we define
\begin{align}\label{not}
x_{ij} := x_i-x_j,\quad v_{ij} := v_i-v_j,\quad r_{ij} := |x_{ij}|,
\end{align}
where $|\cdot|$ denotes the standard $\ell_2$-norm in $\mathbb{R}^d$. Furthermore, for the system of $N$ particles, we use the following abbreviated notation
\begin{align*}
x:=(x_1,...,x_N),\quad v:=(v_1,...,v_N), \quad\mbox{where}\quad x_i,v_i\in \bbr^d.
\end{align*}

Finally, we use the generic harmless constant $C>0$ if the precise control of constants is not beneficial.

\section{Preliminaries}\label{sec:2}
\setcounter{equation}{0}

We begin our considerations by presenting the basic conservation law, energy estimate and a uniform bound on the relative position $r_{ij}=|x_i-x_j|$ for \eqref{B-3}. Since the CSB system \eqref{B-3} is Galilean invariant, we assume, without a loss of generality, that
\begin{align*}
\sum_{i=1}^Nx_i^0 = 0\quad \mbox{and}\quad \sum_{i=1}^Nv_i^0 = 0.
\end{align*}
Then, thanks to the anti-symmetric property of the right-hand side of \eqref{B-3}$_2$, it is easy to see that
\begin{align}\label{vel0}
\frac{d}{dt}\left(\sum_{i=1}^Nx_i\right)=\frac{d}{dt}\left(\sum_{i=1}^Nv_i\right)=0,\quad \mbox{and hence}\nonumber\\
\sum_{i=1}^Nx_i(t) = \sum_{i=1}^Nv_i(t) = 0,\quad \mbox{for all}\ t>0.
\end{align}

Let us present the basic energy estimate for system \eqref{B-3}, with kinetic, potential and total energy defined by
\begin{align*}
{\mathcal E}_k:=\frac{1}{2}\sum_{i=1}^N |v_i|^2,\quad {\mathcal E}_p:=\frac{K_2}{8N}\sum_{i,j=1}^N(r_{ij}-2R)^2,\quad {\mathcal E}:={\mathcal E}_k+{\mathcal E}_p.
\end{align*}
\begin{proposition}\label{P2.1}
Let $(x,v)$ be a smooth solution to system \eqref{B-3} subjected to initial data $(x^0,v^0)$. Then the total energy ${\mathcal E}$ is a non-increasing function with respect to time and the relative distance between the particles is uniformly bounded. More precisely, we have
\begin{align*}
&(1)\quad \frac{d\mathcal{E}(t)}{dt}\le -\frac{K_1}{2N}\sum_{i,j=1}^N\psi(|x_i-x_j|)|v_i-v_j|^2,\\
&(2)\quad \sup_{t\ge0}\sup_{i,j}|x_i(t)-x_j(t)|<2R+\sqrt{\frac{8N{\mathcal E}(0)}{K_2}}=:d_M.
\end{align*}
\end{proposition}
\begin{proof}
The proof is a simplification of the proof of Proposition 1 from \cite{H-K-P} and thus we omit the details. Assertion (1) is obtained through a direct calculation of derivatives of ${\mathcal E}_k$ and ${\mathcal E}_p$ and symmetrization. Then, (2) follows by a direct application of the energy estimate ${\mathcal E}_p\leq {\mathcal E}\leq {\mathcal E}(0)$.
\end{proof}
\begin{remark}\rm
Proposition \ref{P2.1} is valid for system \eqref{B-2}. It also holds regardless of whether the communication weight is singular or regular, since the analysis is based only on the structure of the system.
\end{remark}
\section{Main results}\label{mr}
\setcounter{equation}{0}
In this section, we present our main results. Our results apply to two (four including the original CSB system \eqref{B-2}) frameworks:\\

\noindent $\bullet$~({\bf F}$_1$): Simplified CSB system \eqref{B-3} with a singular communication weight \eqref{sing};\\

\noindent $\bullet$~({\bf F}$_2$): Simplified CSB system \eqref{B-3} with regular communication weight e.g. $\psi(s)=(1+s)^{-\alpha}$ with $\alpha>0$. \\

Both frameworks are considered in $\bbr^d$ on the time interval $[0,\infty)$ subjected to the initial data $(x^0,v^0)$ that are non-collisional, i.e.,
\begin{align*}
x_i^0\neq x_j^0,\qquad\mbox{for all}\ 1\leq i\neq j\leq N.
\end{align*}

\begin{remark}\rm~\\
All of our results remain true also in the case of the original CSB model \eqref{B-2}. The only difference is that one needs to deal with the middle term on the right-hand side of \eqref{B-2} which is non-singular and it does not produce any substantial difficulty.
\end{remark}

\subsection{Collision-avoidance and existence}
\begin{proposition}\label{exi}
Both frameworks {\bf F}$_1$ and {\bf F}$_2$ admit a unique classical solution. In particular, in the singular framework {\bf F}$_1$ particles do not collide in any finite time.
\end{proposition}
\begin{proof}
The existence in framework {\bf F}$_1$ comes directly from the collision avoidance, since the system is regular outside of times of collision. The overall proof of collision avoidance is almost the same as the proof of Theorem 2.1 from \cite{C-C-M-P}. The main difference comes from the bonding force term. The proof in \cite{C-C-M-P} is based on dividing the particles into two groups: Group $A$ of particles that collide with each other and Group $B$ of the remaining particles. Then the singular interactions within $A$ overwhelm the bounded interactions between $A$ and $B$. In the case of the CSB model we simply put the bounded influence of the bonding force together with the interactions between $A$ and $B$, which are dominated by the singular interactions within $A$. The existence in framework {\bf F}$_2$ can be proved by using standard ODE theory.
\end{proof}

\subsection{Asymptotic decay of the kinetic energy}

\begin{proposition}\label{kinetdec}
In both frameworks {\bf F}$_1$ and {\bf F}$_2$, the kinetic energy decays to 0:
\[{\mathcal E}_k(t)\to 0,\quad \mbox{as}\quad t\to\infty.\]
\end{proposition}

\begin{proof}
The following proof is valid in both cases {\bf F}$_1$ and {\bf F}$_2$. We define $\psi_m:=\psi(d_M)$ (with $d_M$ from Proposition \ref{P2.1}) and
{\small
\[\|v\|^2:=\sum_{i=1}^N |v_i|^2=2\mathcal{E}_k \quad\left(=\frac{1}{2N}\sum_{i,j=1}^N|v_i-v_j|^2\quad \mbox{due to}\ \eqref{vel0}\right).\]
}
Then, from Proposition \ref{P2.1}(1), we have
\begin{align}\label{D-4}
\frac{d\mathcal{E}(t)}{dt}\leq -\frac{K_1\psi_m}{2N}\sum_{i,j=1}^N|v_i-v_j|^2 = -K_1\psi_m\|v\|^2,
\end{align}
and thus
\begin{align}\label{D-5}
K_1\psi_m\int_0^\infty \|v(t)\|^2\,dt \leq \mathcal{E}(0).
\end{align}
Thus $\|v(t)\|\to 0$ up to a subsequence. To conclude that actually $\|v(t)\|\to 0$ we require information on uniform-in-time regularity of $\|v\|$. Therefore, we aim to prove that $\|v\|^2$ is uniformly continuous. We multiply $v_i$ to \eqref{B-3}$_2$, sum over all the indices and symmetrize to obtain
\[\frac{1}{2}\frac{d}{dt}\|v\|^2 = -\underbrace{\frac{K_1}{2N}\sum_{i,j=1}^N\psi(r_{ij})|v_{ij}|^2}_{=:P}-\underbrace{\frac{K_2}{4N}\sum_{i,j=1}^N \frac{r_{ij}-2R}{r_{ij}}v_{ji}\cdot x_{ji}}_{=:I}.\]
In the above equation, we recall notation \eqref{not}. By \eqref{D-4} it leads to
\begin{align}\label{D-6}
\frac{1}{2}\frac{d}{dt}\|v\|^2\leq -K_1\psi_m\|v\|^2 + |I|.
\end{align}

Recall that our aim is to prove that $\|v\|^2$ is uniformly continuous. We already know that a part of its derivative $P$ is integrable due to \eqref{D-4} and \eqref{D-5}. Next we show that $I$ is bounded. By Cauchy-Schwarz inequality and boundedness of the relative distance (see Proposition \ref{P2.1}), we have
\begin{equation}\label{D-7}
|I|\leq \frac{K_2}{4N}\sum_{i,j=1}^N|r_{ij}-2R||v_i-v_j| \leq \sqrt{K_2\mathcal{E}(0)} \|v\|.
\end{equation}

However, we also know that $\|v\|= \sqrt{2{\mathcal E}_k}\leq \sqrt{2{\mathcal E}(0)}$. Boundedness of $\|v\|$, together with \eqref{D-7}, implies that $I$ is bounded on $[0,\infty)$. Now we come back to \eqref{D-6} to see that the derivative of $\|v\|^2$ is a sum of an integrable function $P$ and a bounded function $I$. Therefore
\begin{align*}
\|v(t)\|^2 = \|v(0)\|^2-2\underbrace{\int_0^tP\,ds}_{f_1} -2\underbrace{\int_0^t Ids}_{f_2},
\end{align*}
where $f_1$ is absolutely continuous (since its derivative is integrable on $[0,\infty)$) and $f_2$ is Lipschitz continuous. Both absolutely continuous and Lipschitz continuous functions are uniformly continuous and thus $\|v\|^2$ is uniformly continuous. Since $\|v\|^2$ is also integrable by \eqref{D-5}, we conclude that $\|v\|\to 0$ as $t\to \infty$.
\end{proof}

\begin{remark}\rm\label{R4.2}
We note here that the decay of the velocity does not directly imply the convergence of the system toward an equilibrium. If the decay of the velocity is sufficiently slow, the position may not converge, even though the velocity converges to 0. Consider for example $x(t)=\log t$ with $v=\frac{dx}{dt}=\frac{1}{t}$. Although we do not have analytically rigorous proof of convergence of position, we provide numerical evidence of convergence of position in Section \ref{sec:5}.
\end{remark}

\subsection{Global minimal distance between particles}

Next we aim to prove the strict positivity of asymptotic inter-particle distance.
\begin{proposition}\label{anoncoll}
In both frameworks {\bf F}$_1$ and {\bf F}$_2$ there exists $T>0$ and $\rho>0$ such that
\begin{align*}
|x_i(t)-x_j(t)|\geq\rho\quad\mbox{for all}\ t\geq T, \quad i,j=1,...,N.
\end{align*} 
\end{proposition}
We use two convenient observations. The first one, established in \cite{A-C-H-L}, is the following comparability principle for the minimal and maximal inter-particle distance.

\begin{lemma}[\cite{A-C-H-L}]\label{comp}
Let $(x,v)$ be the solution to \eqref{B-2} or \eqref{B-3}. Then there exists a constant $\mu>1$ such that $\max_{i\neq j}r_{ij}\leq\mu\min_{i\neq j}r_{ij}$.
\end{lemma}
\begin{proof}
The proof can be found in \cite{A-C-H-L}, Proposition 5.6, page 641. 
\end{proof}
The second one is a simple observation that if the total energy ${\mathcal E}$ is strictly smaller than $\frac{K_2N}{2}R^2$ at any time $t=T$, then there cannot be any asymptotic (or any finite-time) collisions between the particles after the time $T$. It is an immediate corollary of Proposition \ref{P2.1}. Suppose that there is a collision at the time $t=t_0$. Then, by Lemma \ref{comp}, any collision between two particles implies the total collapse of the particles, which leads to $r_{ij}(t_0)=0$ for all $i,j=1,2,\cdots, N$. Then, from the definition of potential energy, we have
\[\mathcal{E}(t_0)\ge \mathcal{E}_p(t_0)=\frac{K_2N}{2}R^2.\]
Therefore, once the total energy becomes less than $\frac{K_2N}{2}R^2$ at time $T$, we have ${\mathcal E}_p(s)\leq{\mathcal E}(s)\leq{\mathcal E}(T)$ for $s>T$ and the total collapse between the particles is impossible. Again by Lemma \ref{comp}, any collision between particles is impossible and due to the sharpness of ${\mathcal E}(T)<\frac{K_2N}{2}R^2$, the smallest distance between particles is positive in $[T,\infty)$. We summarize these observations:

\begin{description}
\item[(${\mathcal A}$)] There exists a constant $\mu>1$ such that $\max_{i\neq j}r_{ij}\leq\mu\min_{i\neq j}r_{ij}$;
\item[(${\mathcal B}$)] If at any time $T$, we have ${\mathcal E}(T)<\frac{K_2N}{2}R^2$, then there exists $\rho>0$ such that $\min_{i\neq j}r_{ij}\geq\rho$ in $[T,\infty)$.
\end{description}

By observation $({\mathcal A})$ it suffices to show that the asymptotic total collapse of the positions is impossible with $t\to\infty$. Let
\begin{align*}
r(t):=\sqrt{\sum_{i,j=1}^N|x_i-x_j|^2}.
\end{align*}
Note that, using the notation $\|v\|$ introduced in the proof of Proposition \ref{kinetdec}, we have
\[\sum_{i,j=1}^N|v_{ij}|^2=\sum_{i,j=1}^N|v_i-v_j|^2=2N\sum_{i=1}^N|v_i|^2=2N\|v\|^2.\]

Therefore, our goal is to prove that there exists $T>0$ and $\rho>0$, such that for all $t\geq T$ we have $r(t)\geq \rho$. However we begin with a weaker claim that we present in the following lemma.

\begin{lemma}\label{subseq}
In both frameworks {\bf F}$_1$ and {\bf F}$_2$ there exists a sequence $t_n\to\infty$ and $\rho>0$, such that $r(t_n)\geq\rho$. In other words, $r$ does not converge to $0$ with $t\to\infty$. 
\end{lemma}

\begin{proof}
The proof varies depending on the framework. We begin with, what we believe is, a more natural argumentation in the case of regular weight.\\

\noindent $\diamond$ {\bf Regular weight.} In the {\bf F}$_2$ case we assume for simplicity that $\psi(0)=1$. We differentiate $r^2$ with respect to time to get
\begin{align*}
\frac{d}{dt}r^2=2\sum_{i,j=1}^Nx_{ij}\cdot v_{ij}.
\end{align*}
Then we apply \eqref{B-3}$_2$ to find the second derivative of $r^2$:
\begin{align}\label{qA}
\frac{d^2}{dt^2}r^2 &= 2\sum_{i,j=1}^N|v_{ij}|^2 + \frac{2K_1}{N}\sum_{i,j,k}x_{ij}\cdot\Big(v_{ki}\psi(r_{ki})-v_{kj}\psi(r_{kj})\Big) \nonumber\\ 
&+ \frac{K_2}{N}\sum_{i,j,k}x_{ij}\cdot\left(\frac{r_{ki}-2R}{r_{ki}}x_{ki} -\frac{r_{kj}-2R}{r_{kj}}x_{kj}\right)\\
& =: 4N\|v\|^2 + \mathcal{I}_1 + \mathcal{I}_2.\nonumber
\end{align}
Estimation of $\mathcal{I}_1$ and $\mathcal{I}_2$ follow by symmetrization (i.e., by exchanging the indices $i\leftrightarrow k$ and $j\leftrightarrow k$), which leads to
\[\mathcal{I}_1 = -2K_1\sum_{i,j=1}^Nx_{ij}\cdot v_{ij}\psi(r_{ij}),\quad \mathcal{I}_2 = -K_2\sum_{i,j=1}^N r_{ij}(r_{ij}-2R).\]
Therefore we have
\begin{align}\label{rprim}
\frac{d^2}{dt^2}r^2 = 4N\|v\|^2 - 2K_1\sum_{i,j=1}^Nx_{ij}\cdot v_{ij}\psi(r_{ij})\\
 - K_2\sum_{i,j=1}^N r_{ij}(r_{ij}-2R).\nonumber
\end{align}
On the other hand
\begin{align}\label{rbis}
\frac{d^2}{dt^2}r^2 = 2\frac{d}{dt}\left(r\frac{dr}{dt}\right) = 2\left(\frac{dr}{dt}\right)^2 + 2r\frac{d^2r}{dt^2}.
\end{align}
Combining \eqref{rprim} and \eqref{rbis}, together with Cauchy-Schwarz inequalities
\begin{align*}
\left|\frac{dr}{dt}\right|\leq \sqrt{2N} \|v\|,\quad \mbox{and}\quad |\mathcal{I}_1|&\leq 2\sqrt{2N}K_1r\|v\|
\end{align*}
leads to
\begin{align*}
r\frac{d^2r}{dt^2}& =\frac{1}{2}\frac{d^2}{dt^2}r^2-\left(\frac{dr}{dt}\right)^2\ge- K_1\sum_{i,j=1}^Nx_{ij}\cdot v_{ij}\psi(r_{ij})\\
& - \frac{K_2}{2}\sum_{i,j=1}^N r_{ij}^2+K_2R\sum_{i,j=1}^Nr_{ij}\geq - K_1r\|v\| - \frac{K_2}{2}r^2\\
& + K_2R\sum_{i,j=1}^N r_{ij}\ge - K_1r\|v\| - \frac{K_2}{2}r^2 + K_2Rr.
\end{align*}
We divide both sides of the above inequality by $r$ to obtain
\begin{align*}
\frac{d^2r}{dt^2}\geq -K_1\|v\| - \frac{K_2}{2}r + K_2R.
\end{align*}
We know from Proposition \ref{kinetdec} that $\|v(t)\|\to 0$ as $t\to\infty$, thus if $r(t)$ also decays to 0, then for all sufficiently large $t$, we have
\begin{align*}
\frac{d^2r}{dt^2}\geq \frac{K_2}{2}R,
\end{align*}
which is impossible with $r(t)\to 0$. Therefore $r$ does not converge to $0$ as $t\to\infty$ and the assertion of the lemma is proved.\\

\noindent $\diamond$ {\bf Singular weight.} In the {\bf F}$_1$ case the argumentation is different. By the definition of the total energy and noting that $\sum r_{ij}\geq \sqrt{\sum r_{ij}^2}$, we estimate it as
\begin{align}\label{singlem}
{\mathcal E} = {\mathcal E}_k + {\mathcal E}_p = \frac{1}{2}\|v\|^2 + \frac{K_2}{8N}r^2 + \frac{K_2}{2N}R\sum_{i,j=1}^N\Big(R-r_{ij}\Big)\nonumber\\
\leq \frac{1}{2}\|v\|^2 + \frac{K_2}{8N}r^2 +  \frac{K_2}{2N}R(N^2R-r). 
\end{align}

Furthermore, by Proposition \ref{P2.1}, we know that
\begin{align*}
\int_0^\infty \|v\|^2\psi(r)dt= \frac{1}{2N}\int_0^\infty\sum_{i,j=1}^N |v_{ij}|^2\psi(r)dt\\
\le \frac{1}{2N}\int_0^\infty\sum_{i,j=1}^N |v_{ij}|^2\psi(r_{ij})dt\leq {\mathcal E}(0),
\end{align*}
which implies that there exists a sequence $t_n\to\infty$, such that
\begin{align*}
\|v(t_n)\|^2\psi(r(t_n))\to 0.
\end{align*}
If we define $a_n:=\|v(t_n)\|^2\psi(r(t_n))=\|v(t_n)\|^2r^{-\alpha}(t_n)$, we have,
\[a_n\to 0,\quad \mbox{and} \quad\|v(t_n)\|^2 = a_nr^\alpha(t_n).\]

Now assume that $r(t)\to 0$ as $t\to\infty$. Then we come back to \eqref{singlem} to see that
\begin{align*}
{\mathcal E}(t_n)\leq \frac{1}{2}a_nr^\alpha(t_n) + \frac{K_2}{8N}r^2(t_n) +  \frac{K_2}{2N}R(N^2R-r(t_n)).
\end{align*}
Therefore, if we fix sufficiently large $n_0$ so that
\[a_n<\frac{K_2R}{4N},\quad \mbox{and}\quad r(t_n)<1,\quad \mbox{for all $n\ge n_0$},\]
then, 
\begin{align*}
{\mathcal E}(t_{n_0})& < \frac{1}{2}\frac{K_2R}{4N}r^\alpha(t_{n_0})+\frac{K_2R}{8N}r(t_{n_0})+\frac{K_2N}{2}R^2-\frac{K_2R}{2N}r(t_{n_0})\\\
&\leq \frac{1}{2}\frac{K_2R}{4N}r(t_{n_0}) + \frac{K_2R}{8N}r(t_{n_0}) +  \frac{K_2N}{2}R^2-\frac{K_2R}{2N}r(t_{n_0})\\
&\leq \frac{K_2N}{2}R^2-\frac{K_2R}{4N}r(t_{n_0})<\frac{K_2N}{2}R^2.
\end{align*}

Here we note that, since $\alpha\geq 1$ and $r<1$, we have $r^\alpha\leq r$.
Thus, by observation $({\mathcal B})$, it is impossible that $r\to 0$ as $t\to\infty$ and the claim of the lemma is proved.
\end{proof}

\begin{proof}[Proof of Proposition \ref{anoncoll}]

Lemma \ref{subseq} states that regardless of the framework, there exists a subsequence $t_n\to\infty$ and $\rho>0$, such that $r(t_n)\geq\rho$. We may assume without a loss of generality that $\rho\leq R$.

We prove Proposition \ref{anoncoll} by contradiction. Suppose that there exists another subsequence $s_n\to\infty$, such that $r(s_n)\to 0$. It means that ${\mathcal E}_p(s_n)\to \frac{K_2N}{2}R^2$. Thanks to Darboux property, there exists a sequence $q_n\to \infty$ such that $t_n\le q_n\le s_n$ and $R\geq r(q_n)\geq\rho$ (see Figure \ref{Fig0}). Then, by observation $({\mathcal A})$, we have $R\geq r_{ij}(q_n)\geq \frac{1}{\mu N^2}\rho$ since
\[r_{ij}\ge \min_{1\le i\neq j\le N}r_{ij}\ge\frac{1}{\mu}\max_{1\le i\neq j\le N}r_{ij}\ge \frac{r}{\mu N^2}.\]

\begin{figure}[h!]
\centering
\includegraphics[width=0.7\textwidth]{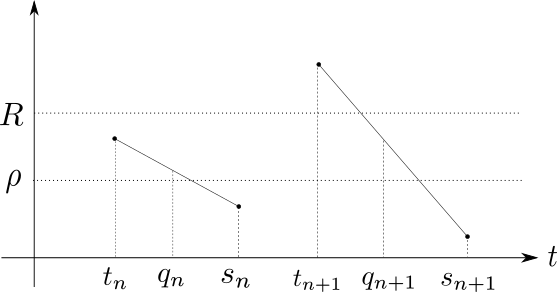}
\caption{Schematic explanation for Darboux property. We introduce $q_n$ because we do not know that $r(t_n)\leq R$.} 
\label{Fig0}
\end{figure}

Then
\begin{align*}
{\mathcal E}(q_n) = {\mathcal E}_k(q_n) + {\mathcal E}_p(q_n)\leq {\mathcal E}_k(q_n) + \frac{K_2N}{8}\left(2R-\frac{\rho}{\mu N^2}\right)^2
\end{align*}
and since ${\mathcal E}_k(q_n)\to 0$ by Proposition \ref{kinetdec}, there exists $n_0$, such that
\begin{align*}
{\mathcal E}(q_{n_0}) \leq \frac{K_2N}{8}\left(2R-\frac{\rho}{2\mu N^2}\right)^2<\frac{K_2N}{2}R^2.
\end{align*}
Then, observation $({\mathcal B})$ implies that there exists $\tilde{\rho}>0$ such that for all $t>q_{n_0}$ and all $i,j=1,...,N$, we have $r_{ij}(t)\geq\tilde{\rho}$. This is contradictory to the assumption of existence of sequence $\{s_n\}$ and the proof of Proposition \ref{anoncoll} is completed.

\end{proof}
We combine Propositions \ref{exi} and \ref{anoncoll} to obtain the following corollary.
\begin{corollary}[Global minimal distance]
In framework {\bf F}$_1$ there exists a global minimal distance between the particles.
\end{corollary}

\subsection{Bound on the relative distance}

\begin{proposition}\label{aspos}
In both frameworks {\bf F}$_1$ and {\bf F}$_2$ we have
\begin{align*}
\limsup_{t\to\infty}|x_i(t)|\leq 2R, \quad\limsup_{t\to\infty}|x_i(t)-x_j(t)|\leq 4R.
\end{align*}
\end{proposition}

\begin{proof}
The lack of asymptotic collisions ensured by Proposition \ref{anoncoll} provides a uniform-in-time regularity of the right-hand side of \eqref{B-3} regardless whether the communication weight is singular or not. Therefore both $v_i$ and $\displaystyle \frac{dv_i}{dt}$ are uniformly continuous and since $v_i\to 0$ as $t\to \infty$, we have
\[0=\lim_{t\to \infty}\frac{dv_i}{dt}=\lim_{t\to\infty}\frac{K_2}{N}\sum_{j=1}^N\frac{1}{2r_{ij}}(r_{ij}-2R)(x_j-x_i).\]

The above equality implies that for any fixed $\varepsilon\ll1$, there exists sufficiently large $t_0$ such that for $t>t_0$,
\begin{equation}\label{C-7}
\left|\sum_{j=1}^N\left[x_j-x_i-\frac{2R}{r_{ij}}(x_j-x_i)\right]\right|<\varepsilon,\quad i=1,2,\cdots, N, \quad t>t_0.
\end{equation}
Together with the zero-sum condition $\displaystyle \sum_{i=1}^Nx_i=0$ (see \eqref{vel0}), \eqref{C-7} implies
\[\left||x_i|-\left|\frac{2R}{N}\sum_{j=1}^N\frac{x_{ji}}{r_{ji}}\right|\right|\le \left|x_i+\frac{2R}{N}\sum_{j=1}^N\frac{x_{ji}}{r_{ji}}\right|<\varepsilon,\quad i=1,2,\cdots, N.\]
Therefore, we have the asymptotic bound for the positions
\[|x_i|\le \left|\frac{2R}{N}\sum_{j=1}^N\frac{x_j-x_i}{|x_j-x_i|}\right| + \varepsilon\le 2R+\varepsilon,\]
and relative distances
\[r_{ij}=|x_i-x_j|\le |x_i|+|x_j|\le 4R+2\varepsilon.\]
Since $\varepsilon$ was arbitrary, this directly implies
\[\limsup_{t\to\infty}|x_i|\le 2R,\quad \mbox{and} \limsup_{t\to\infty}r_{ij}\le 4R.\]
\end{proof}

\section{Numerical Simulation}\label{sec:5}
In this section, we provide the numerical simulation supporting the analytical theorems in this paper. We conduct four simulations: to illustrate the asymptotic bound for position; to compare decay of the energy across four cases of singular or regular original and simplified CSB system; to illustrate the collision avoidance granted by the singularity of the communication weight; and to show that from the numerical perspective the positions of the particles converge to an equilibrium. 

\subsection{Asymptotic bound for position}\label{numpos}
Once the asymptotic bound for position is obtained, it is easy to see that the asymptotic relative distance between particles is at most twice the bound for position. To conduct the numerical simulation, we randomly choose $N=10, 15, 20, 25$ initial positions and velocities from 2 dimensional random vectors uniformly distributed on $[-5,5]\times[-5,5]$ respectively. We also take $R=2$. Figure \ref{Fig1} shows the solutions at $t=500$, at which the solutions are near equilibrium. The large circle is centered at the center of mass with radius $2R$. As expected in Proposition \ref{aspos}, all of the particles stay in that circle, which implies that the radius of position is bounded by $2R$ and their relative distances are bounded by $4R$.

\begin{figure}[h!]
\mbox{
\subfigure[$N=10$]{\includegraphics[width=0.5\textwidth]{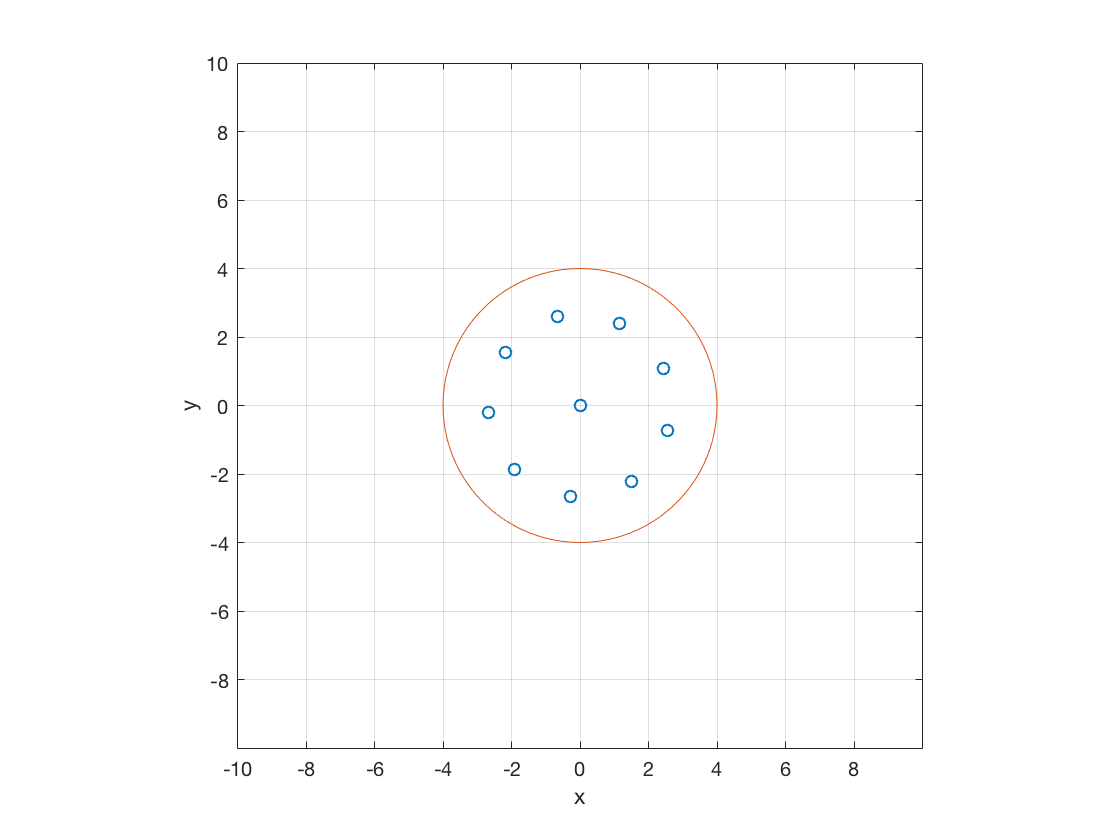}
\label{Fig1a}}
\subfigure[$N=15$]{\includegraphics[width=0.5\textwidth]{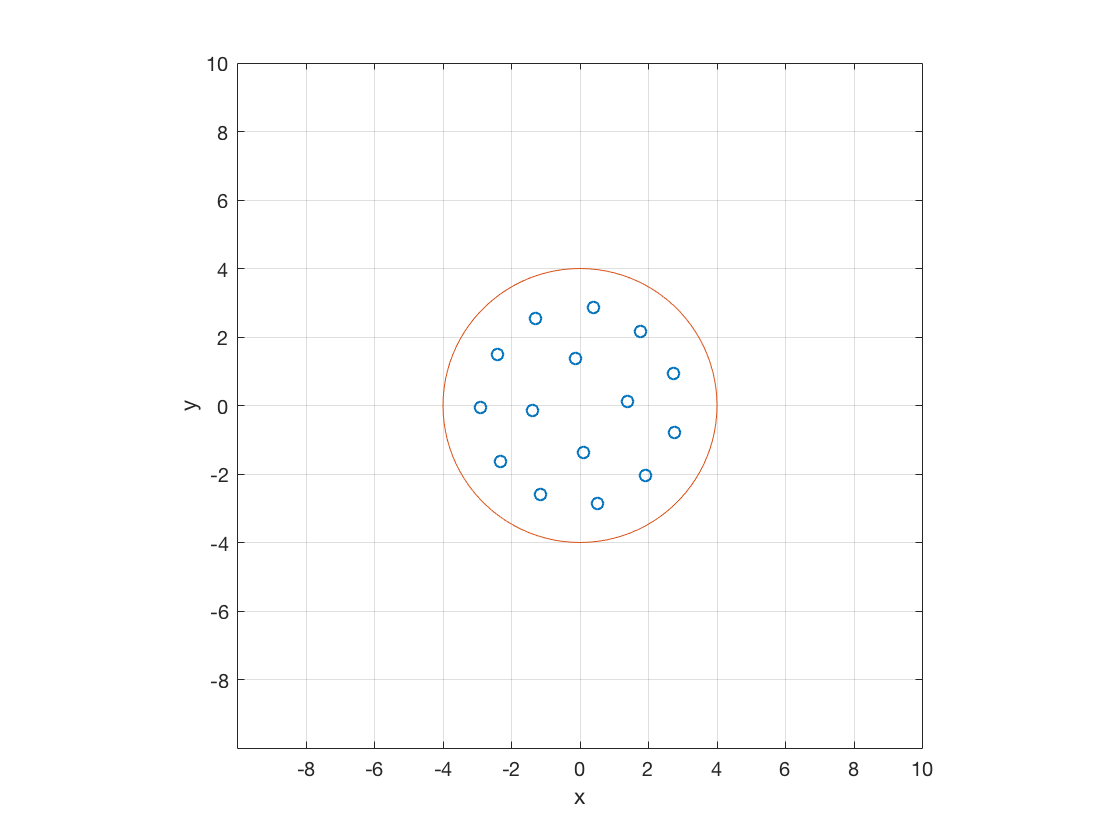}
\label{Fig1b}}
}
\mbox{
\subfigure[$N=20$]{\includegraphics[width=0.5\textwidth]{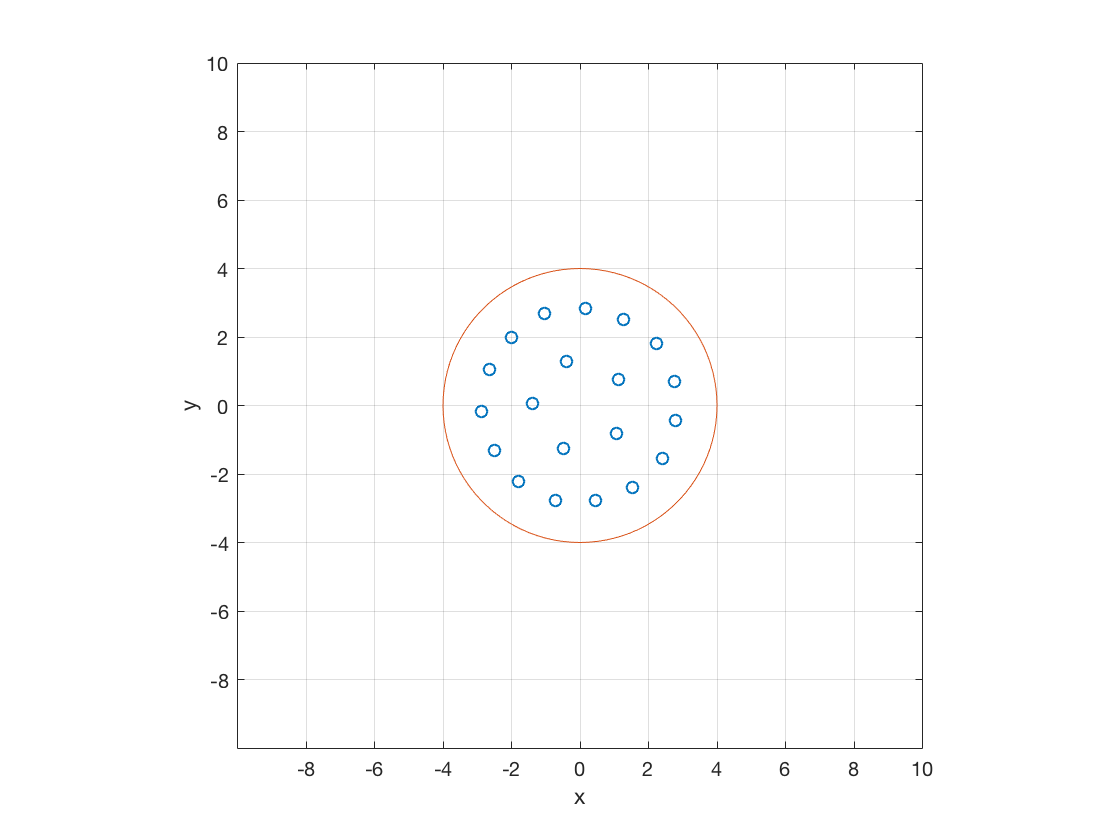}
\label{Fig1c}}
\subfigure[$N=25$]{\includegraphics[width=0.5\textwidth]{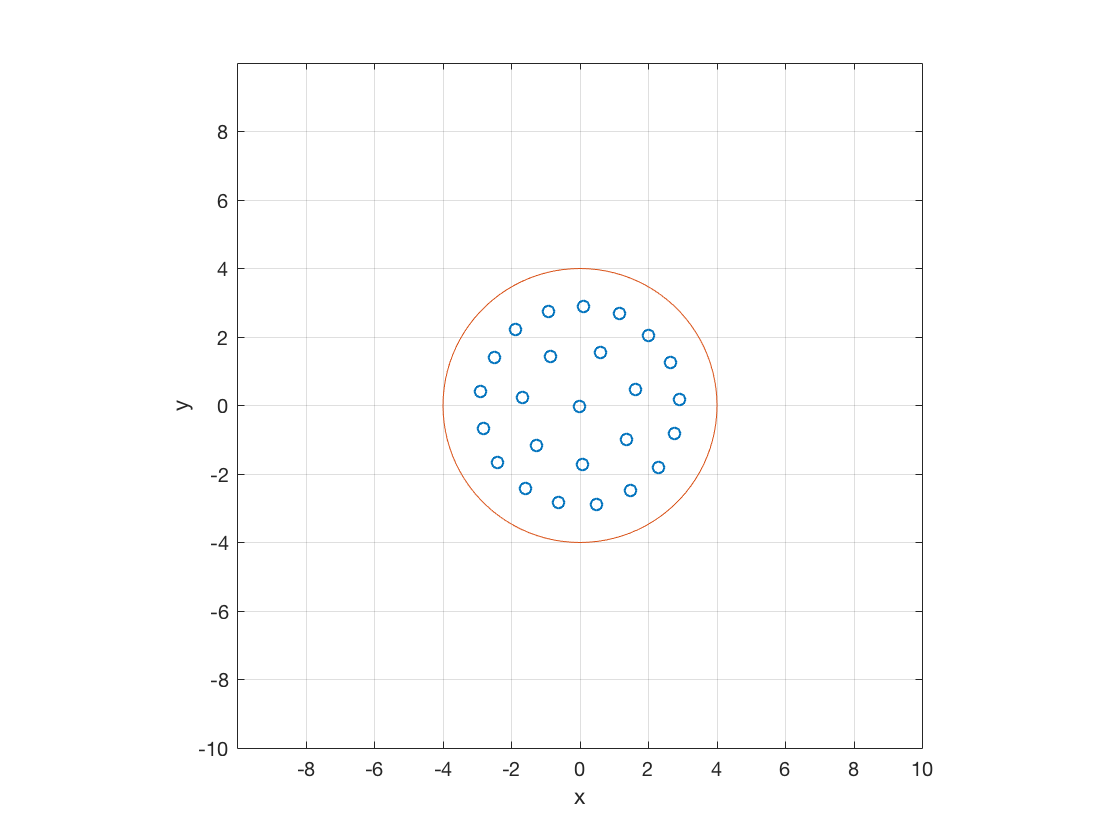}
\label{Fig1d}}
}
\centering\caption{Various pattern formation for different number of particles.}
\label{Fig1}
\end{figure}

\subsection{Energy decay}
Figure \ref{Fig2} presents evolution of the kinetic, potential and total energy for each considered models with the same initial condition. Here, we use the same initial distribution as in Section \ref{numpos}. As expected in Propositions \ref{P2.1} and \ref{subseq}, the kinetic energy ${\mathcal E}_k$ decays to 0 and the total energy converges to a limit value. We note here that the simplified CSB model shows more oscillatory behavior than the original one, which implies that the additional term in original CSB model is somehow related to preventing oscillation.

\begin{figure}[h!]
\mbox{
\subfigure[Original CSB, regular kernel]{\includegraphics[width=0.5\textwidth]{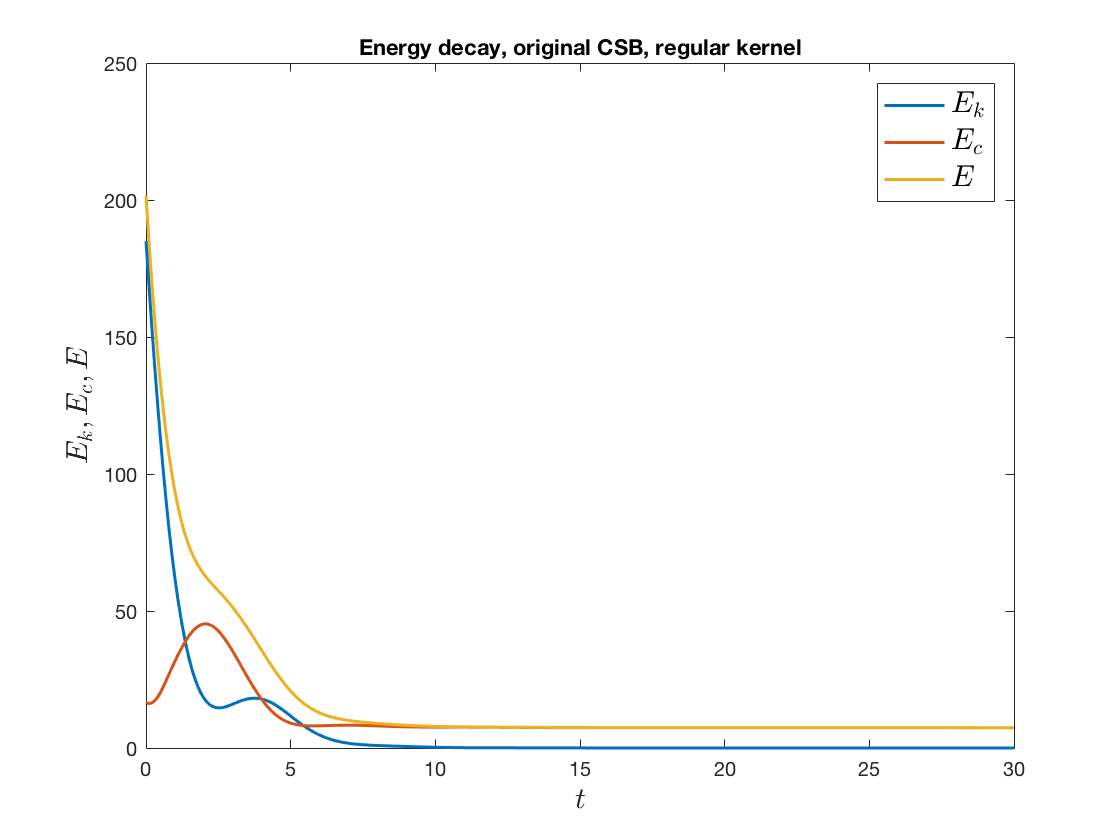}
\label{Fig2a}}
\subfigure[Original CSB, singular kernel]{\includegraphics[width=0.5\textwidth]{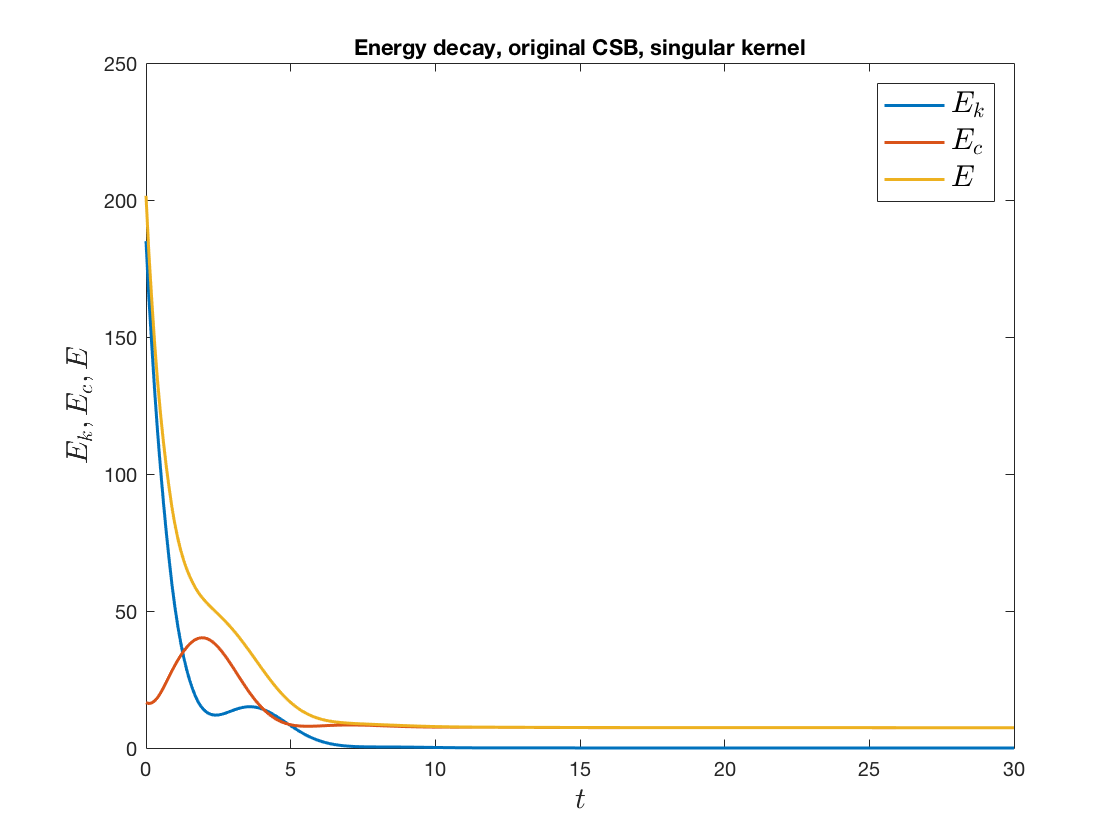}
\label{Fig2b}}
}
\mbox{
\subfigure[Simplified CSB, regular kernel]{\includegraphics[width=0.5\textwidth]{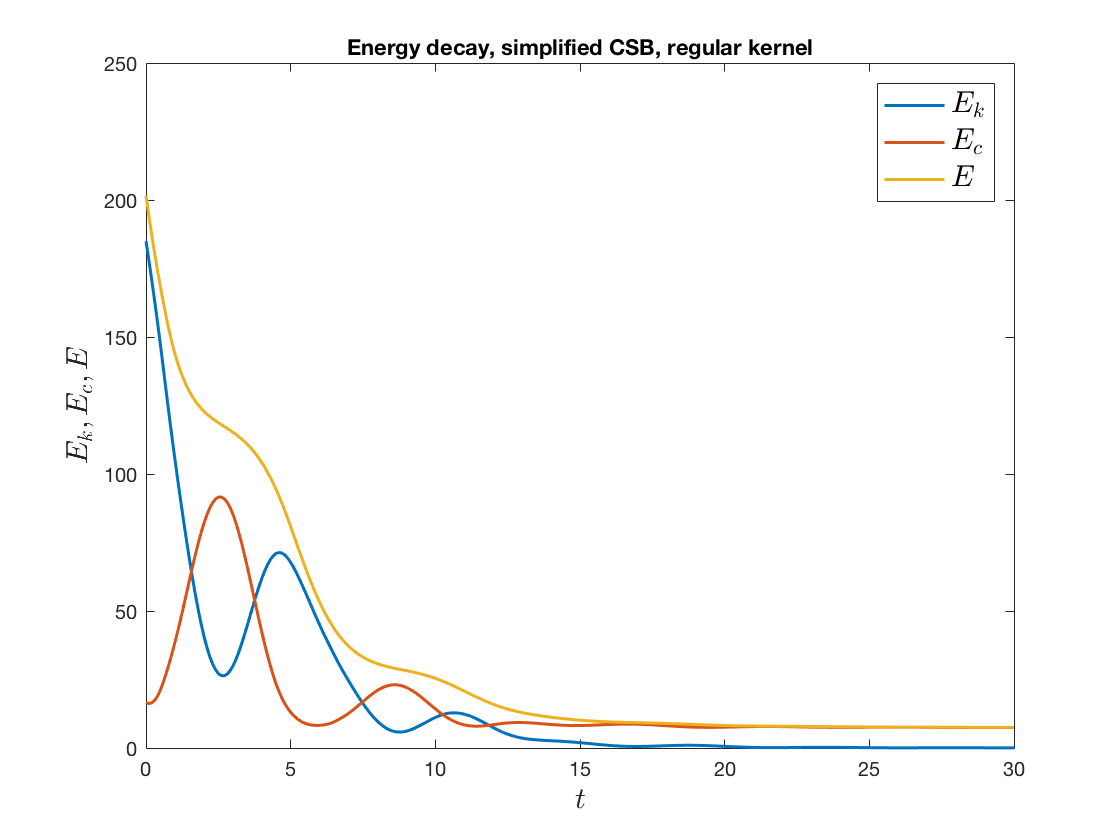}
\label{Fig2c}}
\subfigure[Simplified CSB, singular kernel]{\includegraphics[width=0.5\textwidth]{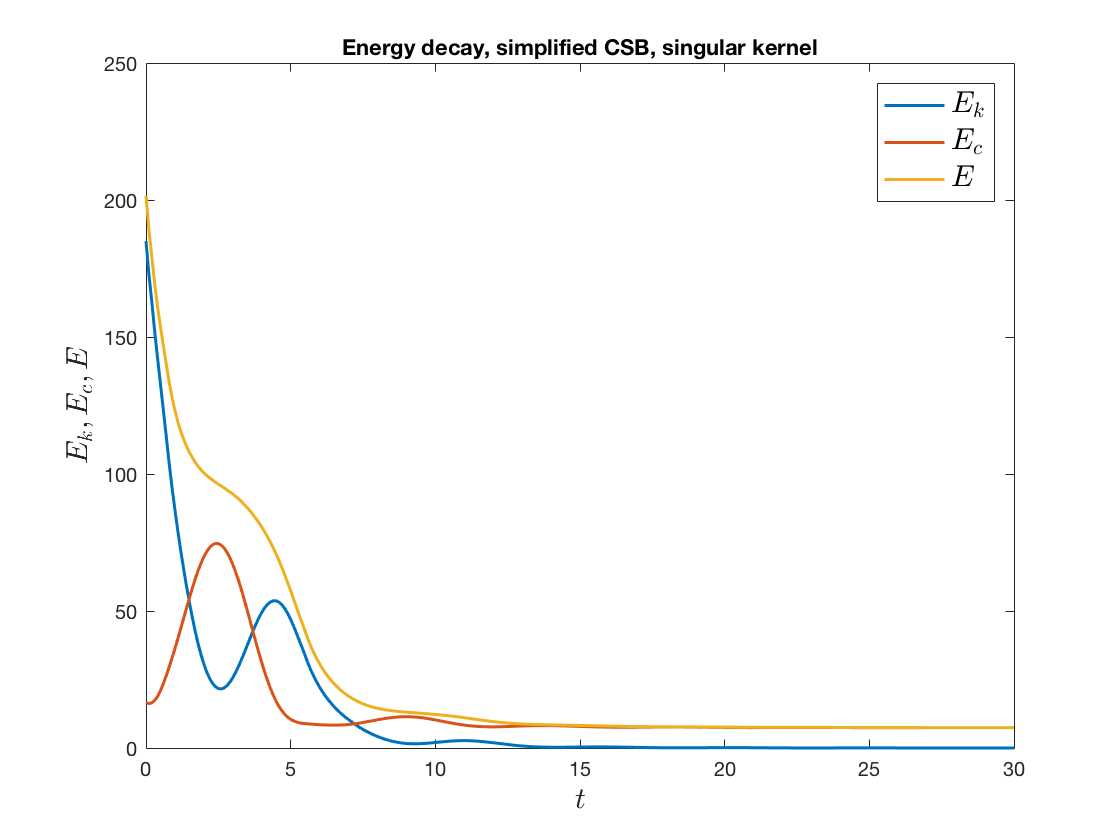}
\label{Fig2d}}
}
\centering\caption{Energy decay of various CSB model.}
\label{Fig2}
\end{figure}

\subsection{Lower bound for distance between particles}
We pick a one-dimensional initial position uniformly from $[-5,5]$ and initial velocity from $[-2,2]$. To compare the singular and regular case, we take two interaction kernels $\psi_{s}$ and $\psi_{r}$ as follows:
\[\psi_{s}(r):=\frac{1}{r},\quad \mbox{and}\quad \psi_{r}(r):=\frac{1}{1+r}.\]

\begin{figure}[h!]
\centering
\mbox{
\subfigure[Particle trajectory for singular kernel]{\includegraphics[width=0.45\textwidth]{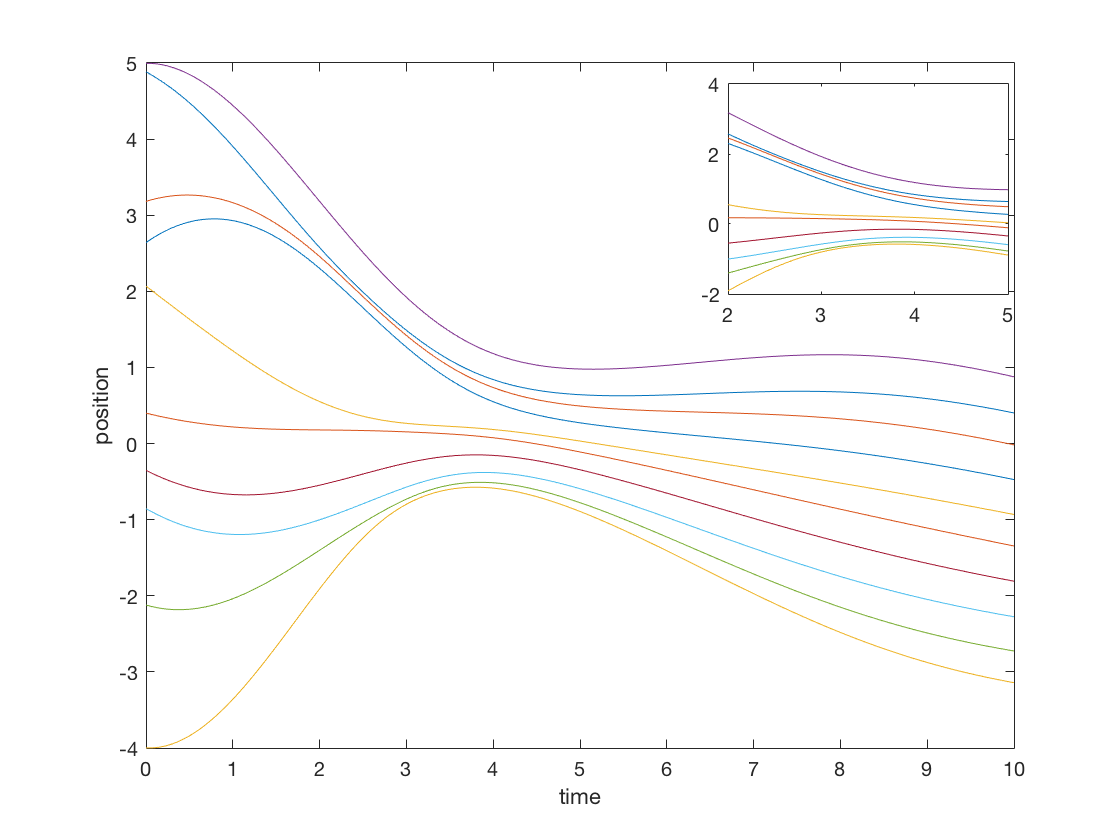}
\label{Fig3a}}
\hspace{0.4cm}
\subfigure[Particle trajectory for regular kernel]{\includegraphics[width=0.45\textwidth]{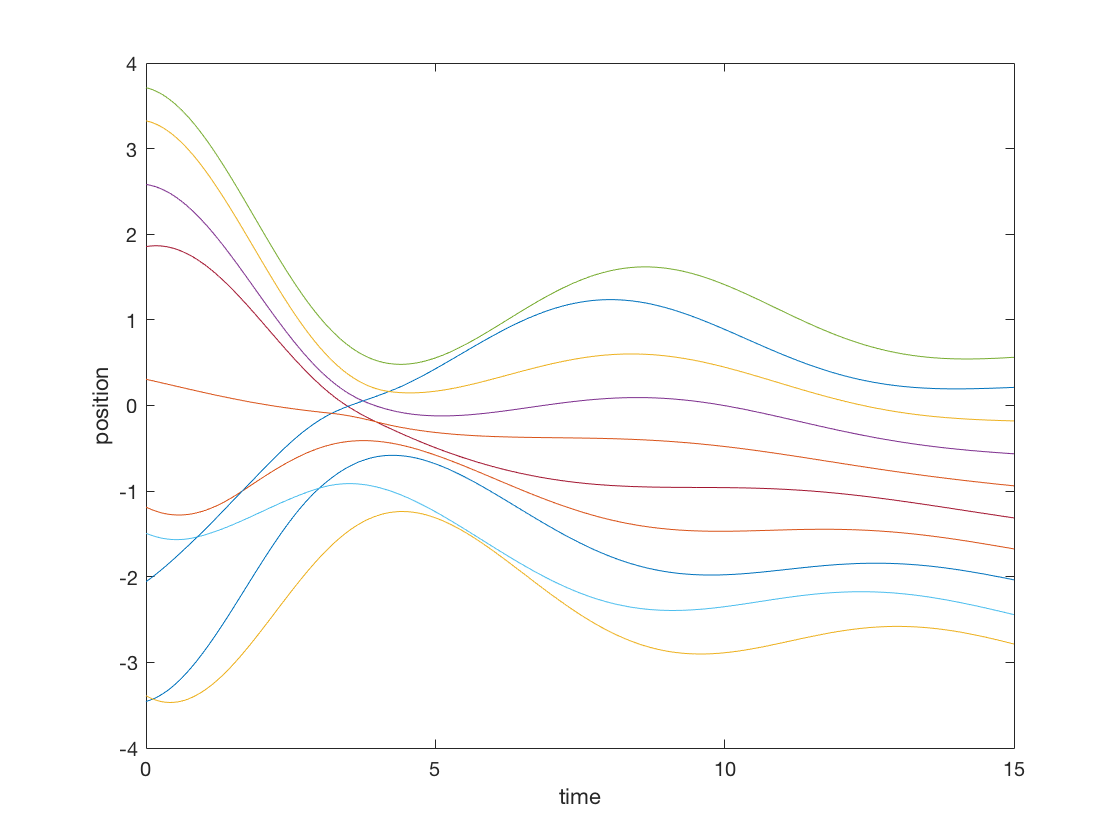}
\label{Fig3b}}
}
\centering\caption{Particle trajectories for (a) singular and (b) regular kernels in 1-dimension.}
\label{Fig3}
\end{figure}

Figure \ref{Fig3a} shows the numerical result for the singular kernel $\psi_{s}$. Here, we find that there is a minimum relative distance attained near $t=3$. However, for the case of regular kernel in Figure \ref{Fig3b}, there are collisions between particles, and the lower bound of relative positions between the particles can only be attained in the asymptotic sense.\\

\subsection{Convergence of position}
Finally we measure the magnitude of velocity to check whether the particles converge to equilibrium. Figure \ref{Fig5} shows the maximum velocity $v_{max}$ defined as
\[v_{max}(t):=\max_{1\le i\le N} |v_i(t)|.\]
In the figure, we multiply $v_{max}$ by $t^{1.5}$ in order to show the clear meaning of the graph. In Figure \ref{Fig5}, one can find that the function $t^{1.5}v_{max}$ is bounded from the above and decays to 0 over time. This implies that there exists a positive constant $C>0$ such that
\[v_{max}\le \frac{C}{t^{1.5}},\]
which further implies the convergence of position, since velocity is absolutely integrable. Therefore, although we cannot provide analytically rigorous proof of convergence of position, we observe that positions of particles converge to the equilibrium.
\begin{figure}[h!]
\centering
\includegraphics[width=0.9\textwidth]{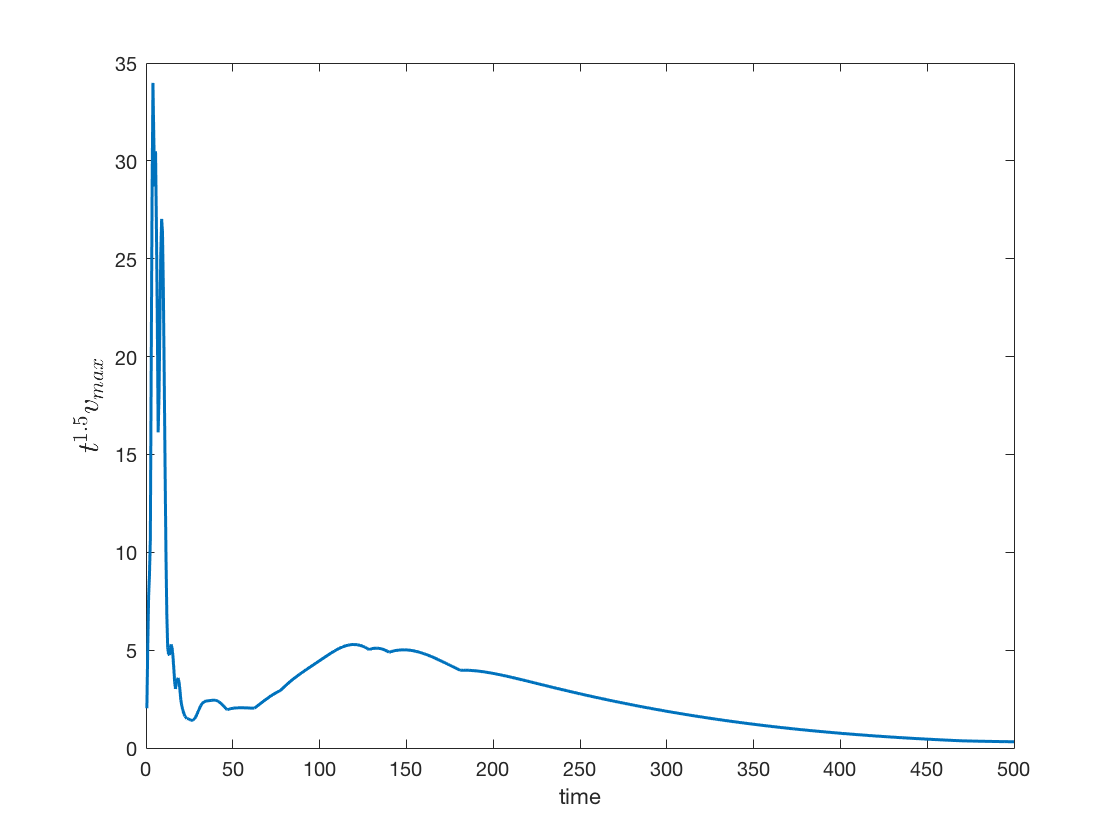}
\caption{Magnitude of $v_{max}$ versus time.} 
\label{Fig5}
\end{figure}

\bigskip

\bibliographystyle{amsplain}
\end{document}